\documentclass{amsart}
\usepackage[english]{babel}
\usepackage[usenames,dvipsnames]{pstricks}
\usepackage{epsfig}

\newtheorem{theorem}{Theorem}

\newtheorem{coro}[theorem]{Corollary}
\newtheorem{remark}[theorem]{Remark}
\newtheorem{definition}[theorem]{Definition}

\newtheorem{example}[theorem]{Example}

\newcommand{\TC}{\mathord{\mathrm{TC}}}

\newcommand{\inside}{\mathord{\mathrm{int}}}
\newcommand{\cl}{\mathord{\mathrm{cl}}}
\newcommand{\cat}{\mathord{\mathrm{cat}}}

\title{Topological complexity of the telescope}
\author[Aleksandra Franc]{Aleksandra Franc}
\address{Institute of Mathematics, Physics and Mechanics\newline\indent Jadranska 19\newline\indent 1000 Ljubljana, Slovenia}
\email{\rm{aleksandra.franc@fmf.uni-lj.si}}
\thanks{The author was supported by the Slovenian Research Agency grant P1-0292-0101.}
\keywords{topological complexity, fibrewise Lusternik-Schnirelmann category, mapping telescope}
\subjclass[2010]{55R70, 55M30}
\begin{document}
\begin{abstract}
We use an alternative definition of topological complexity to show that the topological complexity of the mapping telescope of a sequence $X_1\stackrel{f_1}{\longrightarrow}X_2\stackrel{f_2}{\longrightarrow}X_3\stackrel{f_3}{\longrightarrow}\ldots$ is bounded above by $2\max\{\TC(X_i);\;i=1,2,\ldots\}$.
\end{abstract}
\maketitle

\section{Introduction}

The notion of topological complexity was first introduced by Farber in \cite{Farber:TC}:

\begin{definition}\label{FarberDef}
{\em Topological complexity} $\TC(X)$ of a space $X$ is the least integer $n$ for which there exist an open cover $\{U_1, U_2,\ldots, U_n\}$ of $X\times X$ and sections $s_i\colon U_i\rightarrow X^I$ of the fibration $\pi\colon X^I\rightarrow X\times X$, $\alpha\mapsto (\alpha(0),\alpha(1))$. If no such integer exists we write $\TC(X)=\infty$.
\end{definition}

In \cite{IS}, Iwase and Sakai proved that (for nice spaces $X$) topological complexity is a special case of what James and Morris \cite{JamesMorris} call {\em fibrewise pointed LS category}. A {\em fibrewise pointed space} over a {\em base} $B$ is a topological space $E$, supplied with a {\em projection} $p\colon E\rightarrow B$ and a {\em section} $s\colon B\rightarrow E$. Fibrewise pointed spaces over a base $B$ form a category and the notions of fibrewise pointed maps and fibrewise pointed homotopies are defined as one would expect. More details can be found in \cite{James} and \cite{JamesMorris}.

We consider the product $X\times X$ as a fibrewise pointed space over the base $X$ with the projection to the first component and the diagonal section $\Delta\colon X\rightarrow X\times X$. According to Theorem 1.7 of \cite{IS}, we do not have to work with the fibrewise pointed homotopies but can instead use the less restrictive notion of (unpointed) fibrewise homotopies. A fibrewise homotopy in this case is any homotopy $H\colon X\times X\times I\rightarrow X\times X$ that fixes the first coordinate. So, $H(x,y,t)=(x,h(x,y,t))$ for some homotopy $h\colon X\times X\times I\rightarrow X$. For obvious reasons we call them {\em vertical homotopies}.

We can therefore consider the following theorem as an alternative definition of topological complexity:
\begin{theorem}\label{IwaseDef}
{\em Topological complexity} $\TC(X)$ of a space $X$ is the least integer $n$ for which there exists an open cover $\{U_1, U_2,\ldots, U_n\}$ of $X\times X$ such that each $U_i$ is vertically compressible to the diagonal $\Delta(X)$. If no such integer exists we write $\TC(X)=\infty$.
\end{theorem}

Note that we do not require the homotopies to be stationary on the section $\Delta(X)$, nor do we require the sets $U_i$ to contain the section.

Our result is analogous to the statement concerning LS category proven by Ganea in \cite{Ganea}. He gave an example to show that the LS category of the telescope is not necessarily equal to the LS categories of its parts. As we will see, this is also true for topological complexity. In \cite{Hardie}, Hardie improved Ganea's bound by 1 and Ganea's example shows that Hardie's bound is sharp.

\section{Topological complexity of the telescope}

We approach the problem indirectly by first estimating the topological complexity of an increasing union. The increasing union is much easier to handle and we can explicity construct a cover with the required properties. We then use homotopy invariance of topological complexity to apply the result to mapping telescopes.

\begin{theorem}\label{maintheorem}
Let $X=\bigcup_{i=1}^{\infty}X_i$ be the increasing union of closed subspaces with the property that for each $i$ there exists an open set $Y_i\subset X$ such that $X_i\subset Y_i\subset\cl(Y_i)\subset \inside{(X_{i+1})}$. If $\TC(X_i)\leq n$ for all $i$, then $\TC(X)\leq 2n$.
\end{theorem}

\begin{proof}
Since $X_i\subset X_{i+1}$ for all $i$, we have $X_i\times X_i\subset X_{i+1}\times X_{i+1}$ for all $i$ and the product $X\times X=\bigcup_{i=1}^{\infty}X_i\times X_i$ is an increasing union of its subspaces. Let $\{U_j^{(i)}\}_{j=1}^n$ be an open cover of $X_i\times X_i$ with sets $U_j^{(i)}$ vertically compressible to the diagonal $\Delta(X_i)\subset\Delta(X)$.
Define $L_i=\inside(X_i\times X_i)-\cl(Y_{i-2}\times Y_{i-2})$ for $i>2$, $L_2=\inside(X_2\times X_2)$, $L_1=\inside(X_1\times X_1)$. Here, $\inside(A)$ and $\cl(A)$ denote the interior and the closure of $A$ as a subset of $X\times X$. Let $V_j^{(i)}=U_j^{(i)}\cap L_i$ and consider the sets
$$W_1 = \bigcup_{i=1}^\infty V_1^{(2i-1)}, W_2 = \bigcup_{i=1}^\infty V_1^{(2i)},\ldots, W_{2n-1} = \bigcup_{i=1}^\infty V_n^{(2i-1)}, W_{2n} = \bigcup_{i=1}^\infty V_n^{(2i)}.$$
Figure \ref{figure} illustrates the construction of the first three sets from $W_1$.

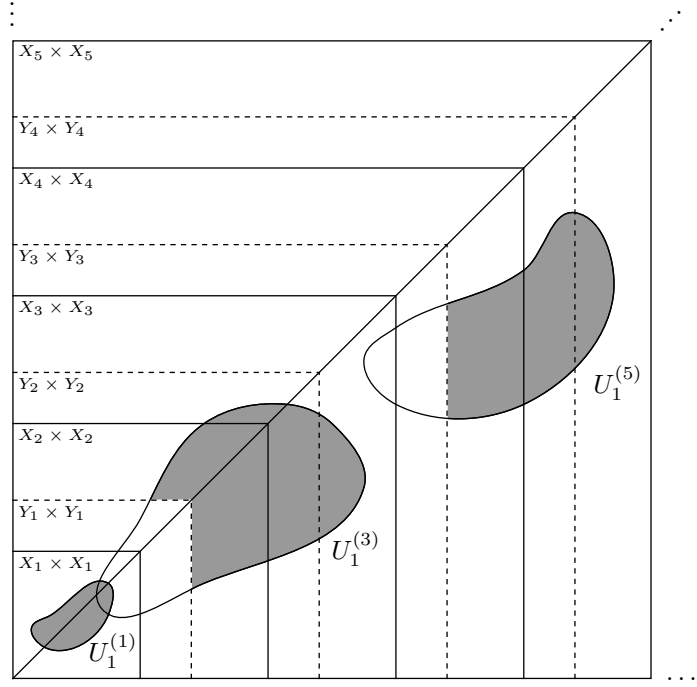
\begin{figure}[ht]
\begin{center}
\begin{pspicture}(0,0)(8.5,8.8)
\psset{linewidth=0.5pt,dash=2pt 2pt,unit=0.85cm}
\psset{labelsep=3pt}
\definecolor{siva}{gray}{1}
\definecolor{svsiva}{gray}{0.6}
\definecolor{tesiva}{gray}{0}
\psccurve[fillstyle=solid,fillcolor=svsiva](5.5,5)(6,4.3)(8,4.3)(9.4,6)(8.8,7.3)(8,6.4)(6,5.5)
\pspolygon[fillstyle=solid,fillcolor=siva,linecolor=siva](6.8,0)(6.8,6.8)(0,6.8)
\psccurve[fillstyle=solid,fillcolor=svsiva](1.4,1.1)(3,1.5)(5.5,3)(5,4)(4.3,4.3)(3,4)(2,2.5)
\pspolygon[fillstyle=solid,fillcolor=siva,linecolor=siva](2.8,0)(2.8,2.8)(0,2.8)(1,0)
\psccurve[fillstyle=solid,fillcolor=svsiva](0.5,0.5)(1,0.5)(1.5,1)(1.5,1.5)(0.6,1)(0.3,0.8)
\psline[linestyle=dashed,linecolor=tesiva](2.8,0)(2.8,2.8)(0,2.8)
\psline[linestyle=dashed,linecolor=tesiva](4.8,0)(4.8,4.8)(0,4.8)
\psline[linestyle=dashed,linecolor=tesiva](6.8,0)(6.8,6.8)(0,6.8)
\psline[linestyle=dashed,linecolor=tesiva](8.8,0)(8.8,8.8)(0,8.8)
\psframe(0,0)(10,10)
\psline(0,0)(10,10)
\psline(2,0)(2,2)(0,2)
\psline(4,0)(4,4)(0,4)
\psline(6,0)(6,6)(0,6)
\psline(8,0)(8,8)(0,8)
\uput[dr](0,2){\tiny $X_1\times X_1$}
\uput[dr](0,4){\tiny $X_2\times X_2$}
\uput[dr](0,6){\tiny $X_3\times X_3$}
\uput[dr](0,8){\tiny $X_4\times X_4$}
\uput[dr](0,10){\tiny $X_5\times X_5$}
\uput[dr](0,2.8){\tiny $\color{tesiva}Y_1\times Y_1$}
\uput[dr](0,4.8){\tiny $\color{tesiva}Y_2\times Y_2$}
\uput[dr](0,6.8){\tiny $\color{tesiva}Y_3\times Y_3$}
\uput[dr](0,8.8){\tiny $\color{tesiva}Y_4\times Y_4$}
\uput[dr](1.1,0.8){$U_1^{(1)}$}
\uput[dr](4.9,2.4){$U_1^{(3)}$}
\uput[dr](9,5){$U_1^{(5)}$}
\psccurve(1.4,1.1)(3,1.5)(5.5,3)(5,4)(4.3,4.3)(3,4)(2,2.5)
\psccurve(5.5,5)(6,4.3)(8,4.3)(9.4,6)(8.8,7.3)(8,6.4)(6,5.5)
\psccurve(0.5,0.5)(1,0.5)(1.5,1)(1.5,1.5)(0.6,1)(0.3,0.8)
\uput[u](0,10.1){$\vdots$}
\uput[r](10.1,0){$\ldots$}
\rput{45}(10.3,10.3){$\ldots$}
\end{pspicture}
\end{center}
\caption{The shaded areas represent the sets $V_1^{(1)}$, $V_1^{(3)}$ and $V_1^{(5)}$. These sets are all part of $W_1$.}
\label{figure}
\end{figure}

We observe the following:
\begin{itemize}
\item Every $(x,y)\in X$ belongs to $L_i$ for some $i$ and is therefore contained in $V_j^{(i)}$ for some $j$. So, $\{W_k\}_{k=1}^{2n}$ covers $X\times X$.
\item Each $V_j^{(i)}$ can be compressed to $\Delta(X_i)\subset\Delta(X)$ by the restriction of the vertical homotopy defined on $U_j^{(i)}$.
For all positive integers $l$ and $m$ we have $L_{l}\cap L_{m}=\emptyset$ as long as $|l-m|\geq 2$, so $V_j^{(l)}\cap V_j^{(m)}=\emptyset$ for $|l-m|\geq 2$. 
The vertical homotopies we defined on $V_j^{(i)}$ can therefore be combined to define a (continuous) homotopy that vertically compresses $W_k$ to $\Delta(X)$.
\item The sets $L_i$ are open in $X\times X$, so $V_j^{(i)}=U_j^{(i)}\cap L_i$ are open in $L_i$ and therefore in $X\times X$. Each $W_k$ is defined as a union of open sets, so all $W_k$ are open.
\end{itemize}
From this we infer that $\{W_k\}_{k=1}^{2n}$ is indeed an open cover of $X\times X$ with each $W_k$ vertically compressible to $\Delta(X)$. The conclusion now follows from Theorem \ref{IwaseDef}.
\end{proof}

\begin{remark}
The proof of Theorem \ref{maintheorem} can be reused with only minor alterations to notation to prove a slightly more general statement. For a fibrewise pointed space $p\colon E\rightarrow B$ with section $s$ denote by $\cat_B^*(E)$ the {\em fibrewise unpointed category} as in Definition 1.6 of \cite{IS}. Assume that $E=\bigcup_{i=1}^{\infty}E_i$ is an increasing union of closed subspaces with the property that $s(p(E_i))\subset E_i$ and that there exist open sets $Y_i\subset E$ such that $E_i\subset Y_i\subset\cl{(Y_i)}\subset\inside{(E_{i+1})}$. Let $B_i=p(E_i)$ and denote by $p_i\colon E_i\rightarrow B_i$ the restriction of $p$ to $E_i$ with the section $s_i$ being the restriction of section $s$ to $B_i$. If $\cat_{B_i}^*(E_i)\leq n$, then $\cat_{B}^*(E)\leq 2n$.
\end{remark}

We now represent a mapping telescope as an increasing union of subspaces and obtain the following result:

\begin{coro}\label{maincoro}
Let $X=\bigcup_{i=1}^{\infty}X_i\times[i-1,i]$ be the mapping telescope of a sequence of maps $$X_1\stackrel{f_1}{\longrightarrow}X_2\stackrel{f_2}{\longrightarrow}X_3\stackrel{f_3}{\longrightarrow}\ldots$$
and let $\TC(X_i)\leq n$ for all $i$. Then $\TC(X)\leq 2n$.
\end{coro}

\begin{proof}
Define $X'_n=\bigcup_{i=1}^{n}X_i\times[i-1,i]$ to be the union of the first $n$ mapping cylinders in the telescope $X=\bigcup_{i=1}^{\infty}X_i\times[i-1,i]$. Then $X$ is the increasing union $X=\bigcup_{i=1}^{\infty}X'_i$ and we can take
$$Y_i=\left(\bigcup_{i=1}^{n}X_i\times[i-1,i]\right)\cup X_{i+1}\times\left[i,i+1/2\right).$$
Since $X'_i$ are homotopy equivalent to $X_i$ for all $i$, we have $\TC(X'_i)=\TC(X_i)\leq n$ for all $i$. The conclusion now follows from Theorem~\ref{maintheorem}.
\end{proof}

Finally, here is an equivalent formulation of Corollary \ref{maincoro}:

\begin{coro}
Let $X=\bigcup_{i=1}^{\infty}X_i\times[i-1,i]$ be the mapping telescope of a sequence of maps $$X_1\stackrel{f_1}{\longrightarrow}X_2\stackrel{f_2}{\longrightarrow}X_3\stackrel{f_3}{\longrightarrow}\ldots.$$
Then $\TC(X)\leq 2\max\{\TC(X_i);\;i=1,2,\ldots\}$.
\end{coro}

\begin{proof}
If $\TC(X_i)$ are not bounded above, then $\max\{\TC(X_i);\;i=1,2,\ldots\}=\infty$ and the statement is trivially true. If $\max\{\TC(X_i);\;i=1,2,\ldots\}=M<\infty$, then $\TC(X_i)\leq M$ for all $i$ and Corollary~\ref{maincoro} implies that $\TC(X)\leq 2M$.
\end{proof}

\begin{example}
The mapping telescope of the sequence
$$S^1\stackrel{\cdot 2}{\longrightarrow}S^1\stackrel{\cdot 2}{\longrightarrow}S^1\stackrel{\cdot 2}{\longrightarrow}\ldots$$
is $X=K(\mathbb{Z}[\frac{1}{2}],1)$. We have $\TC(S^1)=2$ and Corollary \ref{maincoro} implies that $\TC(X)\leq 4$. The cohomology of $X$ is nontrivial only in dimension 2, and there we have $H^2(X;\mathbb{Z})=\hat{\mathbb{Z}}_2/\mathbb{Z}$, where $\hat{\mathbb{Z}}_2$ denotes the group of $2$-adic integers (detailed calculations can be found in \cite{Hatcher}, Section 3F, in particular Example 3F.9). Elements of finite order in $\hat{\mathbb{Z}}_2/\mathbb{Z}$ are represented by rational numbers. Since $\hat{\mathbb{Z}}_2/\mathbb{Z}$ is uncountable, there exists an element $u\in H^2(X;\mathbb{Z})$ of infinite order and we obtain a non-trivial product of length $2$:
$$(1\otimes u-u\otimes 1)^2=-2u\otimes u\in H^2(X;\mathbb{Z})\otimes H^2(X;\mathbb{Z}).$$
Combining Theorem 7 of \cite{Farber:TC} and Theorem 4 of \cite{Schwarz} we get a lower bound in terms of zero-divisors: $\TC(X)\geq 3$.
So, $3\leq\TC(X)\leq 4$.

Notice how in this example our upper bound is better than the standard upper bounds in terms of dimension and LS category (see \cite{Farber:TC}, Theorem 4 and Theorem 5), although it is not low enough to determine $TC(X)$.
\end{example}
This example shows that the topological complexity of the telescope $X$ can be greater than the topological complexity of its parts $X_i$. The question remains of whether our bound can be improved by 1.

\end{document}